\documentclass[12pt, twoside]{article}
\usepackage{amsmath,amsthm,amssymb}
\usepackage{times}
\usepackage{enumerate}
\usepackage{bm}
\usepackage[all]{xy}
\usepackage{mathrsfs}
\usepackage{amscd}
\usepackage{color}
\def\titlerunning#1{\gdef\titrun{#1}}
\makeatletter
\def\author#1{\gdef\autrun{\def\and{\unskip, }#1}\gdef\@author{#1}}
\def\address#1{{\def\and{\\\hspace*{18pt}}\renewcommand{\thefootnote}{}%
		\footnote {#1}}%
	\markboth{\autrun}{\titrun}}
\makeatother
\def\email#1{e-mail: #1}

\def\keywords#1{\par\medskip
	\noindent\textbf{Keywords.} #1}

\newtheorem{theorem}{Theorem}[section]
\newtheorem{corollary}[theorem]{Corollary}

\newtheorem{proposition}[theorem]{Proposition}
\theoremstyle{definition}
\newtheorem{definition}[theorem]{Definition}
\newtheorem{remark}[theorem]{Remark}

\numberwithin{equation}{section}

\xyoption{all}

\def \C {\mathbb{C}}

\def \R {\mathbb{R}}

\def \R {\mathbb{R}}

\def \a {\alpha }
\def \b {\beta}

\def \De {\Delta}
\def \la {\lambda}
\def \La {\Lambda}
\def\w {\omega}
\def\Om{\Omega}

\begin{document}
\baselineskip=17pt

\titlerunning{On second non-HLC degree of closed symplectic manifold }
\title{On second non-HLC degree of closed symplectic manifold}

\author{Teng Huang}

\date{}

\maketitle

\address{Teng Huang: School of Mathematical Sciences, University of Science and Technology of China; CAS Key Laboratory of Wu Wen-Tsun Mathematics,  University of Science and Technology of China, Hefei, Anhui, 230026, People’s Republic of China; \email{htmath@ustc.edu.cn;htustc@gmail.com}}
\begin{abstract}
In this note, we show that for a closed almost-K\"{a}hler manifold $(X,J)$ with the almost complex structure $J$ satisfies $\dim\ker P_{J}=b_{2}-1$ the space of de Rham harmonic forms is contained in the space of symplectic-Bott-Chern harmonic forms. In particular, suppose that $X$ is four-dimension, if the self-dual Betti number $b^{+}_{2}=1$, then we prove that the second non-HLC degree measures the gap between the de Rham and the symplectic-Bott-Chern harmonic forms. 
\end{abstract}
\section{Introduction}
Let $X$ be a closed $2n$-dimensional smooth manifold and denote by $\Om^{k}(X)$ the space of smooth $k$-forms on $X$. Suppose that $X$ admits a symplectic structure $\w$; then many cohomology groups can be defined on $(X,\w)$. In \cite{TY} Tseng-Yau noticing that the de Rham cohomology is not the appropriate cohomology to talk about symplectic Hodge theory, define a symplectic version of the Bott-Chern and the Aeppli cohomology groups. The symplectic Bott-Chern cohomology groups and the symplectic Aeppli cohomology groups defined respectively as
$$H^{k}_{d+d^{\La}}=\frac{\ker(d+d^{\La})\cap\Om^{k}(X)}{{\rm{Im}}dd^{\La}\cap\Om^{k}(X)},$$
and
$$H^{k}_{dd^{\La}}:=\frac{\ker(dd^{\La})\cap\Om^{k}(X)}{({\rm{Im}}d+{\rm{Im}}d^{\La})\cap\Om^{k}(X)}.$$
Suppose that $J$ is an $\w$-compatible almost complex structure, i.e., $J^{2}=-id$, $\w(J\cdot, J\cdot)=\w(\cdot,\cdot)$, and $g(\cdot,\cdot)=\w(\cdot, J\cdot)$ is a Riemannian metric on $X$. The triple $(\w, J, g)$ is called an almost K\"{a}hler structure on $X$. Notice that any one of the pairs $(\w,J)$, $(J,g)$ or $(g,\w)$ determines the other two. An almost-K\"{a}hler metric $(\w,J,g)$ is K\"{a}hler if and only if $J$ is integrable. 

There are canonical isomorphisms
$$\mathcal{H}^{k}_{d+d^{\La}}(X):=\ker\De_{d+d^{\La}}\cong H^{k}_{d+d^{\La}}(X)$$
and 
$$\mathcal{H}^{k}_{dd^{\La}}(X):=\ker\De_{dd^{\La}}\cong H^{k}_{dd^{\La}}(X),$$
where 
$$\De_{d+d^{\La}}:=(dd^{\La})(dd^{\La})^{\ast}+\la(d^{\ast}d+d^{\La_{\ast}}d^{\La}),$$
$$\De_{dd^{\La}}=(dd^{\La})^{\ast}(dd^{\La})+\la(dd^{\ast}+d^{\La}d^{\La_{\ast}}),$$
are fourth-order elliptic self-adjoint differential operators and $\la>0$. In particular, the symplectic cohomology groups are finite-dimensional vector spaces on a compact symplectic manifold. For $\sharp\in\{d+d^{\La},dd^{\La}\}$, we denote $h^{\bullet}_{\sharp}:=\dim H^{\bullet}_{\sharp}<\infty$ when the manifold $X$ is understood.

In \cite{AT3,TT} Angella and Tomassini, starting from a purely algebraic point of view, introduce on a compact symplectic manifold $(X^{2n},\w)$ the following non-negative integers
$$\De_{s}^{k}(X):=h^{k}_{d+d^{\La}}(X)+h^{k}_{dd^{\La}}(X)-2b_{k}(X)\geq0, \forall\ k\in\mathbb{Z},$$
proved that, similarly to the complex case, their vanishing characterizes the $dd^{\La}$-lemma. As already observed in \cite{AT3} we can write the non-HLC degrees as follows
$$\De_{s}^{k}=2(h^{k}_{d+d^{\La}}-b_{k}), \forall\ k\in\mathbb{Z}.$$
Therefore, one has that for all $k=1,\cdots,n$,
$$b_{k}\leq h^{k}_{d+d^{\La}}$$
on a compact symplectic $2n$-dimensional maifold.  Moreover the equalities 
$$b_{k} =h^{k}_{d+d^{\La}},\ \forall\ k\in\mathbb{Z},$$
hold on a compact symplectic $2n$-dimensional manifold if and only if it satisfies the Hard-Lefschetz condition \cite{TT2}. In fact, by \cite[Corollary 2]{Mat}, \cite[Theorem 0.1]{Yan}, it turns out that the following  conditions are equivalent:\\
(1) $X$ satisfies $dd^{\La}$-lemma,\\
(2) the natural homomorphism $H^{\bullet}_{d+d^{\La}}(X)\rightarrow H^{\bullet}_{dR}(X)$ is actually an isomorphism,\\
(3) every de Rham cohomology class admits a representative being both $d$-closed and $d^{\La}$-closed,\\
(4) the Hard Lefschetz Condition holds on $X$.
 
In this note, we consider the second non-HLC degree on closed almost K\"{a}hler manifold. If the space of de Rham harmonic forms is contained in the space of symplectic-Bott-Chern harmonic forms, i.e., $\mathcal{H}^{2}_{dR}(X)\subset\mathcal{H}^{2}_{d+d^{\La}}(X)$, then identity $b_{2}\leq h^{2}_{d+d^{\La}}$ naturally holds. In the fifth section of \cite{Lej}, Lejmi introduced the differential operator $P_{J}$ on a closed almost K\"{a}hler 4-manifold $(X,J)$ Tan-Wang-Zhou \cite{TWZ2} extended the defined to higher dimensions. We can give a sufficient condition for $\mathcal{H}^{2}_{dR}(X)\subset\mathcal{H}^{2}_{d+d^{\La}}(X)$. Furthermore, in four-dimension case, we prove that $\dim\ker{P}_{J}=b_{2}-1$ is the sufficient and necessary condition for $\mathcal{H}^{2}_{dR}(X)\subset\mathcal{H}^{2}_{d+d^{\La}}(X)$.
\begin{theorem}\label{T1}
Let $(X,J)$ be a $2n$-dimensional closed almost K\"{a}hler manifold, $n\geq2$. If the almost complex structure $J$ satisfies that $\dim\ker{P_{J}}=b^{2}-1$, then $$\mathcal{H}^{2}_{dR}(X)=\mathcal{H}^{+}_{J}(X)\oplus\mathcal{H}^{-}_{J}(X)\subset\mathcal{H}^{2}_{d+d^{\La}}(X).$$
In particular, $b_{2}\leq h^{2}_{d+d^{\La}}$.
\end{theorem}
\begin{remark}
Tomassini \cite{Tom} proved that if every class in $H^{+}_{J}(X)$ has a $J$-invariant harmonic representative and the almost complex $J$ is $C^{\infty}$-pure and full, then $\mathcal{H}_{dR}^{2}\subset\mathcal{H}^{2}_{d+d^{\La}}$. In fact, if every class in $H^{+}_{J}(X)$ has a $J$-invariant harmonic representative, we then get  $\dim H^{+}_{J}(X)\leq \dim\mathcal{H}^{+}_{J}(X)$.  Therefore $H^{+}_{J}(X)\cong \mathcal{H}^{+}_{J}(X)$ since $\mathcal{H}^{+}_{J}(X)\subset H^{+}_{J}(X)$. If $J$ is also $C^{\infty}$-pure and full, i.e., $H^{2}_{dR}(X)=H^{+}_{J}(X)\oplus H^{-}_{J}(X)$, then $\mathcal{H}^{2}_{dR}(X)=\mathcal{H}^{+}_{J}(X)\oplus\mathcal{H}^{-}_{J}(X)$. Following Theorem \ref{T1}, we also obtain the result proved by Tomassini.
\end{remark}
\begin{corollary}(\cite{Tom})
Let $(X,J)$ be a $2n$-dimensional closed almost K\"{a}hler manifold, $n\geq2$, suppose that $J$ is $C^{\infty}$-pure and full. Assume that every class in $H^{+}_{J}(X)$ has a $J$-invariant harmonic representative. Then
$$\mathcal{H}^{2}_{dR}(X)\subset\mathcal{H}^{2}_{d+d^{\La}}(X).$$
In particular, $b_{2}\leq h^{2}_{d+d^{\La}}$.
\end{corollary}
Suppose $(X,J)$ is a $4$-dimensional closed almost K\"{a}hler manifold. Then the Hodge star $\ast$ gives the well-known self-dual, anti-self-dual decomposition of $2$-forms:
\begin{equation}\nonumber
\Om^{2}=\Om^{+}_{g}\oplus\Om^{-}_{g},\ \a=\a^{+}_{g}+\a^{-}_{g}.
\end{equation}
Let $b_{2}$ be the second Betti number, and $b_{2}^{\pm}$ be the self-dual, respectively, anti-self-dual Betti number of the $4$-dimensional $X$.
\begin{theorem}\label{T2}
Let $(X,J)$ be a four-dimensional closed almost K\"{a}hler manifold. The following conditions are equivalent:\\
(1) $\mathcal{H}_{dR}^{2}(X)=\mathcal{H}^{+}_{J}(X)\oplus\mathcal{H}^{-}_{J}(X)$,\\
(2) $\dim\ker{P_{J}}=b_{2}-1$, i.e., $h^{-}_{J}=b^{+}_{2}-1$,\\
(3) $\mathcal{H}^{2}_{dR}(X)\subset\mathcal{H}^{2}_{d+d^{\La}}(X)$.
\end{theorem}
By Hodge theory, the Hodge star $\ast_{g}$ induces cohomology decomposition by the metric $g$:
$$H^{2}(X,\mathbb{R})\cong\mathcal{H}^{2}_{dR}(X,\mathbb{R})=\mathcal{H}^{+}_{g}\oplus\mathcal{H}^{-}_{g}.$$
It's easy to see that $H^{-}_{J}\subset\mathcal{H}^{+}_{g}$ and $\mathcal{H}^{-}_{g}\subset H^{+}_{J}$. In \cite[Proposition 3.1]{DLZ}, they proved that if $J$ is  almost K\"{a}hler, then
$h^{+}_{J}\geq b_{2}^{-}+1$, $h^{-}_{J}\leq b_{2}^{+}-1$. When $b^{+}_{1}=0$, we have $h^{-}_{J}=0$.
\begin{corollary}
	Let $(X,J)$ be a four-dimensional closed almost K\"{a}hler manifold. If $b^{+}_{2}(X)=1$, then
	$$\mathcal{H}^{2}_{dR}(X)=\mathbb{R}\w\oplus\mathcal{H}^{-}_{g}(X)\subset\mathcal{H}^{2}_{d+d^{\La}}(X).$$ 
	In particular, $b_{2}\leq h^{2}_{d+d^{\La}}$.
\end{corollary}
\section{Definitions and Preliminaries}
We recall some definitions and results on the differential forms for almost complex and almost Hermitian manifolds. Let $X$ be a $2n$-dimensional manifold (without boundary) and $J$ be a smooth almost complex structure on $X$. The space  $\Om^{k}(X)$ of real smooth differential $k$-forms has a type decomposition:
$$\Om^{k}(X)=\bigoplus_{p+q=k}\Om^{p,q}_{J}(X),$$
where 
$$\Om^{p,q}_{J}(X)=\{\a\in\Om^{p,q}_{J}(X,\C)\oplus\Om^{q,p}_{J}(X,\C):\a=\bar{\a}\}.$$
For a finite set $S$ of pairs of integers, let
$$\mathcal{Z}^{S}_{J}=\bigoplus_{(p,q)\in S}\mathcal{Z}^{p,q}_{J},\ \mathcal{B}^{S}_{J}=\bigoplus_{(p,q)\in S}\mathcal{B}^{p,q}_{J},$$
where 
$$\mathcal{Z}^{p,q}_{J}:=\{\a\in\Om_{J}^{(p,q),(q,p)}:d\a=0\}$$ 
and 
$$\mathcal{B}^{p,q}_{J}:=\{\b\in\Om_{J}^{(p,q),(q,p)}:there\ exists\ \gamma\ such\ that\ \b=d\gamma\}.$$
Denoting by $\mathcal{B}$ the space of $d$-exact forms, we have that
$$\frac{\mathcal{Z}^{S}_{J}}{\mathcal{B}^{S}_{J}}=\frac{\mathcal{Z}^{S}_{J}}{\mathcal{B}\cap\mathcal{Z}^{S}_{J}}= \frac{\mathcal{Z}^{S}_{J}}{\mathcal{B}}.$$
Here, there is a natural inclusion
$$\rho_{S}:\frac{\mathcal{Z}^{S}_{J}}{\mathcal{B}^{S}_{J}}\rightarrow\frac{\mathcal{Z}}{\mathcal{B}}.$$
As in \cite{LZ}, we will write $\rho_{S}({\mathcal{Z}^{S}_{J}}/{\mathcal{B}^{S}_{J}})$ simply as ${\mathcal{Z}^{S}_{J}}/{\mathcal{B}^{S}_{J}}$ and we may define the cohomology spaces
$$H^{S}_{J}(X):=\{[\a]:\a\in\mathcal{Z}^{S}_{J} \}=\frac{\mathcal{Z}^{S}_{J}}{\mathcal{B}},$$
In particular, there is a natural inclusion
$$H^{1,1}_{J}(X)+H^{(2,0),(0,2)}_{J}(X)\subset H^{2}(X),$$
but the sum can be neither direct nor equal to $H^{2}_{dR}(X,\R)$.
The almost complex structure $J$ acts on the space $\Om^{2}$ of smooth $2$-forms on $X$ as an involution by
\begin{equation}\label{E9}
\a\longmapsto\a(J\cdot,J\cdot),\ \a\in\Om^{2}(X).
\end{equation}
This gives the $J$-invariant, $J$-anti-invariant decomposition of $2$-forms
\begin{equation}\label{E1}
\Om^{2}=\Om^{+}_{J}\oplus\Om^{-}_{J},\ \a=\a^{+}_{J}+\a^{-}_{J}.
\end{equation}
Specifically, if $k=2$, $J$ acts on $\Om^{2}(X)$ as (\ref{E9}) and decomposes it into the topological direct sum of the invariant part $\Om^{+}_{J}$ and the anti-invariant part $\Om^{-}_{J}$. In this case, the two decompositions are related in the following way:
\begin{equation*}
\begin{split}
&\Om^{+}_{J}(X)=\Om^{1,1}_{J}(X,\R):=\Om_{J}^{1,1}(X,\C)\cap\Om^{2}(X), \\
&\Om^{-}_{J}(X)=\Om_{J}^{(2,0),(0,2)}(X,\R):=(\Om_{J}^{2,0}(X,\C)\oplus\Om_{J}^{0,2}(X,\C))\cap\Om^{2}(X). \\
\end{split}
\end{equation*}
We also use the notation $\mathcal{Z}^{2}$ for the space of closed $2$-forms on $X$ and $\mathcal{Z}^{\pm}_{J}=\mathcal{Z}^{2}\cap\Om^{\pm}_{J}$ for the corresponding projections.  Define the $J$-invariant, $J$-anti-invariant cohomology subgroups $H_{J}^{\pm}$ by \cite{LZ}
$$H_{J}^{\pm}=\{\mathfrak{a}\in H^{2}(X,\R):\exists\a\in\mathcal{Z}_{J}^{\pm}\ such\ that\ [\a]=\mathfrak{a}\}.$$
Therefore we recall the following definition
\begin{definition}(\cite[Definition 2.2, 2.3, Lemma 2.2]{LZ})
	An almost complex structure $J$ on a closed symplectic manifold $X$ is called
	-- $C^{\infty}$-pure if  $$H^{+}_{J}(X)\cap H^{-}_{J}(X)=\{0\},$$
	-- $C^{\infty}$-full if $$H_{dR}^{2}(X)=H_{J}^{+}(X)\cap H^{-}_{J}(X),$$
	-- $C^{\infty}$-pure and full if $$H_{dR}^{2}(X)=H_{J}^{+}(X)\oplus H^{-}_{J}(X).$$
\end{definition}
Dr\v{a}ghici-Li-Zhang \cite{DLZ} have proved that any closed four-dimensional manifold endowed with an almost-complex structure $J$ satisfies the decomposition $H^{2}_{dR}(X;\R)=H^{+}_{J}(X;\R)\oplus H^{-}_{J}(X;\R)$ which can be regarded as a Hodge decomposition for non-K\"{a}hler $4$-manifolds. This decomposition does not hold true in higher dimension \cite{AT1,AT2,FT}. The decomposition of $H^{2}_{dR}(X,\R)$ is known to be ture for integrable almost structures $J$ that admit compatible K\"{a}hler metrics on compact manifolds of any dimension. In this case, this is nothing but the classical real Dolbeault decomposition of $H^{2}(X,\R)$ \cite{ATZ,DLZ,FT,HMT,LZ}. In \cite{DLZ2}, they made a conjecture about the dimension $h_{J}^{-}$ of $H^{-}_{J}$ on a compact $4$-manifold which asserts that $h_{J}^{-}$ vanishes for $4$-manifolds for generic almost complex structures $J$. They also proved this conjecture for $4$-manifolds with $b_{2}^{+}=1$ \cite[Theorem 3.1]{DLZ2}. Tan-Wang-Zhang-Zhu confirmed the conjecture completely by using $g$-compatible almost complex structures \cite[Theorem 1.1]{TWZZ}.
\section{Proof of main results}
The Lefschetz operator  $L:\Om^{k}(X)\rightarrow\Om^{k+2}(X)$ defined by
$$L(\a)=\w\wedge\a.$$
It has adjoint $\La=\ast^{-1}L\ast$. There is a Lefschetz decomposition on complex $k$-forms
$$\Om^{k}(X)=\bigoplus_{r\geq0}L^{r}P^{k-2r},$$
where $P^{\bullet}=\ker{\La}\cap\Om^{\bullet}(X)$. 

We consider the following second order linear differential operator on $P^{2}:=\ker\La\cap\Om^{2}$,
\begin{equation*}
\begin{split}
P_{J}:&P^{2}\rightarrow P^{2}\\
&\psi\mapsto\De_{d}\psi-\frac{1}{n}g(\De_{d}\psi,\w)\w.\\
\end{split}
\end{equation*}
where $\De_{d}$ is the Riemannian Laplacian with respect to the metric $g(\cdot, \cdot)=\w(\cdot, J\cdot)$	(here we use the convention $g(\w,\w)= n$). By studying Lejmi's operator $P_{J}$ \cite{Lej}, Tan-Wang-Zhou \cite{TWZ2} proved that $J$ is $C^{\infty}$-pure and full when $\dim (\ker P_{J})=b_{2}-1$. 
\begin{theorem}(\cite[Theorem 2.5]{TWZ2})
	Suppose that $(X,g,J,\w)$ is a closed almost K\"{a}hler $2n$-manifold, $n\geq2$. If $\dim (\ker P_{J})=b^{2}-1$, then $J$ is $C^{\infty}$-pure and full and 
	\begin{equation*}
	H^{2}_{dR}(X,\mathbb{R})=H^{+}_{J}\oplus H^{-}_{J}=\mathbb{R}\w\oplus H^{+}_{J,0}\oplus H^{-}_{J}.
	\end{equation*}
	where $$H^{+}_{J,0}=\{\mathfrak{a}\in H^{2}_{dR}(X,\R):\ there\ exists\ \a\in\mathcal{Z}^{2}\cap\Om^{+}_{J,0}\ such\ that\ \mathfrak{a}=[\a] \}.$$
\end{theorem}
We then have a decomposition of the space of harmonic $2$-forms.
\begin{proposition}\label{P5}
	Suppose that $(X,g,J,\w)$ is a closed almost K\"{a}hler $2n$-manifold, $n\geq2$. If $\dim (\ker P_{J})=b_{2}-1$, then
	\begin{equation*}
	\mathcal{H}^{2}_{dR}(X,\mathbb{R})=\mathcal{H}^{1,1}_{J}(X,\mathbb{R})\oplus\mathcal{H}^{(2,0),(0,2)}_{J}(X,\mathbb{R}).
	\end{equation*}
\end{proposition}
\begin{proof}
	Let $\a\in\mathcal{H}^{1,1}_{J}(X,\mathbb{R})$. We denote $\a=f\w+\a^{1,1}_{0}$, where $f$ is a function on $X$, $\a_{0}^{1,1}\in P^{1,1}$, i.e., $\La\a^{1,1}_{0}=0$.  Following Weil formula,
	$$\ast\gamma^{+}_{J}=\ast(f\w+\a_{0}^{1,1})=f\wedge\frac{\w^{n-1}}{(n-1)!}-\a^{1,1}_{0}\wedge\frac{\w^{n-2}}{(n-2)!}.$$
	Noting that $d\gamma^{+}_{J}=0$, i.e., $df\wedge\w+d\a^{1,1}_{0}=0$, we then have
	$$0=L^{n-1}(df)+L^{n-2}(d\a^{1,1}_{0}).$$
	Noting that $d\ast\gamma^{+}_{J}=0$. We also have
	$$0=\frac{1}{(n-1)!}L^{n-1}(df)-\frac{1}{(n-2)!}L^{n-2}(d\a^{1,1}_{0}).$$
	Combining preceding identities, it implies that
	$$L^{n-1}(df)=L^{n-2}(d\a^{1,1}_{0})=0.$$
	Since the map $L^{n-k}:\Om^{k}\rightarrow\Om^{2n-k}$ is bijective for $k\leq n$, $df=0$ and hence $d\a^{1,1}_{0}=0$. We then have $$\mathcal{H}^{1,1}_{J}(X,\mathbb{R})=\mathbb{R}\w\oplus\mathcal{H}^{+}_{J,0}\cong\mathbb{R}\w\oplus{H}^{+}_{J,0}.$$
	Therefore, 
	$$\mathcal{H}^{1,1}_{J}\oplus\mathcal{H}^{(2,0),(0,2)}_{J}\subset\mathcal{H}^{2}_{dR}\cong H^{2}_{dR}=\mathbb{R}\w\oplus H^{+}_{J,0}\oplus H^{-}_{J}\cong\mathcal{H}^{1,1}_{J}\oplus\mathcal{H}^{(2,0),(0,2)}_{J}.$$
\end{proof}
We denote by 
$$\mathcal{H}^{\pm}_{J}(X):=\{\a\in\Om^{\pm}_{J}(X):\De_{d}\a=0 \}$$
the spaces of harmonic $J$-invariant forms and $J$-anti-invariant forms.
\begin{proposition}\label{P1}
Let $(X,J)$ be a $2n$-dimensional closed almost K\"{a}hler manifold, $n\geq2$. We have a decomposition of the space of harmonic $2$-forms as follows 
$$\mathcal{H}_{dR}^{2}(X)=\mathcal{H}^{+}_{J}(X)\oplus\mathcal{H}^{-}_{J}(X),$$
if and only if $\dim\ker{P_{J}}=b_{2}-1$.
\end{proposition}
\begin{proof}
By hypothesis on the space of harmonic $2$-forms, we then have $$\mathcal{H}_{dR}^{2}(X)=\mathbb{R}\w\oplus\mathcal{H}^{+}_{J,0}(X)\oplus\mathcal{H}^{-}_{J}(X).$$	
Therefore $\dim\ker{P}_{J}=\dim\mathcal{H}^{+}_{J,0}(X)+\dim\mathcal{H}^{-}_{J}(X)=b_{2}-1$.
\end{proof}
\begin{proof}[\textbf{Proof of Theorem \ref{T1}}]
The conclusions follow from Proposition \ref{P1}.	
\end{proof}
\begin{remark} By definition $\De_{s}^{2}:=2(h^{2}_{d+d^{\La}}-b_{2})$ and so on a compact almost-K\"{a}hler manifold with $J$ satisfies  $\dim\ker{P_{J}}=b^{2}-1$ we just proved, with a different technique, that $\De_{s}^{2}\geq0$.
\end{remark}
\begin{proof}[\textbf{Proof of Theorem \ref{T2}}]
{$(1)\Leftrightarrow(2)$\\
The conclusion follows Proposition \ref{P1}.}\\
$(1)\Longrightarrow (3)$\\
Let $\a\in\mathcal{H}^{2}(X)$, namely $d\a=0$ and $d\ast\a=0$.  By hypothesis, $\mathcal{H}^{2}_{dR}(X)=\mathbb{R}\w\oplus\mathcal{H}^{-}_{g}(X)\oplus\mathcal{H}^{-}_{J}(X)$, then
$$\a=c\w+\a^{-}_{g}+\gamma^{-}_{J},$$
where $c$ is a constant, $\a^{-}_{g}\in\mathcal{H}^{-}_{g}(X)$ and $\gamma^{-}_{J}\in\mathcal{H}^{-}_{J}(X)$. We only need to show that $d^{\La}\a=0$ since $(dd^{\La})\a=d^{\La_{\ast}}d^{\ast}\a=0$. In fact $d^{\La}\a=d\La(c\w)=0$.\\
$(3)\Rightarrow(2)$\\
Let $\a\in\mathcal{H}^{+}_{g}(X,\mathbb{R})$, namely $\ast\a=\a$ and $d\a=0$. Following {the idea of \cite[Theorem 4.2]{TT}}, if $\mathcal{H}^{2}_{dR}\subset\mathcal{H}^{2}_{d+d^{\La}}$, where $$\mathcal{H}^{2}_{d+d^{\La}}=\ker{d}\cap\ker{d^{\La}}\cap\ker(dd^{\La})^{\ast}\cap\Om^{2},$$
we then have 
$$0=d^{\La}\a=-\ast J^{-1}dJ\ast\a=-\ast J^{-1}dJ\a.$$ 
We consider the unique decomposition
$$\a=f\w+\gamma^{-}_{J},$$ 
where $f$ is a function and $\gamma^{-}_{J}$ is a $J$-anti-invariant form, hence
$$J\a=f\w-\gamma^{-}_{J}.$$
Noting that 
$$0=d\a=df\wedge\w+d\gamma^{-}_{J}.$$
and
$$0=dJ\a=df\wedge\w-d\gamma^{-}_{J}.$$
Thus $df\wedge\w=0$, that is $df=0$ and $d\gamma^{-}_{J}=0$. Hence $$\mathcal{H}_{g}^{+}(X,\mathbb{R})\cong \mathbb{R}\w\oplus\mathcal{H}^{-}_{J}\cong \mathbb{R}\w\oplus H^{-}_{J},$$
i.e., $h^{-}_{J}=b^{+}_{2}-1$. Since $h^{+}_{J}+h^{-}_{J}=b_{2}$ and $b^{+}_{2}+b^{-}_{2}=b_{2}$, we have {$h^{+}_{J}=b^{-}_{2}+1$}.
\end{proof}
\begin{remark}
For any $\a\in\Om^{2}(X)$, we write 
$$\a=f\w+\a_{0}^{+}+\a_{J}^{-},$$
where $f$ is a function on $X$, $\a_{0}^{+}\in\Om^{+}_{J,0}(X)$ and $\a_{J}^{+}\in\Om^{-}_{J}(X)$.

In higher dimensional case, if $\mathcal{H}^{2}_{dR}\subset\mathcal{H}^{2}_{d+d^{\La}}$, we obtain that 
$$0=d^{\La}\a=[d,\La]\a=df.$$
Noting that
$$\ast\a=-\frac{1}{(n-1)!}L^{n-1}f-\frac{1}{(n-1)!}L^{n-2}\a^{+}_{0}+\frac{1}{(n-2)!}L^{n-2}\a^{-}_{J}.$$
Therefore, 
$$0=d\a=d\a^{+}_{0}+d\a^{-}_{J}$$
and
$$0=d\ast \a=-\frac{1}{(n-1)!}L^{n-2}d\a^{+}_{0}+\frac{1}{(n-2)!}L^{n-2}d\a^{-}_{J}.$$
Hence, we get
$$L^{n-2}d\a^{+}_{0}=L^{n-2}d\a^{-}_{J}=0,$$
i.e., $d^{\ast}\a^{+}_{0}=d^{\ast}\a^{-}_{J}=0$.
\end{remark}
Let $(X^{2n}, J)$ be a compact almost-K\"{a}hler manifold and suppose that $J$ satisfies $\dim\ker P_{J}=b_{2}-1$. In view of Theorem \ref{T1} we denote the with $V$ the finite dimensional vector spaces of $\mathcal{H}^{2}_{dR}(X)$ such that
$$\mathcal{H}^{2}_{dR}(X)\oplus V=\mathcal{H}^{2}_{d+d^{\La}}(X).$$
It follows that 
$$\dim V= h^{2}_{d+d^{\La}}-b_{2}=\frac{1}{2}\De_{s}^{2},$$
namely, $\frac{1}{2}\De_{s}^{2}$ can be seen as the dimension of a vector subspace of the space of harmonic forms $H^{2}_{dR}(X)$. Now we describe explicitly the space $V$.
\begin{proposition}
Let $(X^{2n}, J)$ be a compact almost-K\"{a}hler manifold and suppose that $J$ satisfies $\dim\ker P_{J}=b_{2}-1$. Then,
$$V=\mathcal{H}^{2}_{d+d^{\La}}(X)\cap{\rm{Im}}d.$$
\end{proposition}
\begin{proof}
By definition of $V$
$$\mathcal{H}^{2}_{dR}(X)\oplus V=\mathcal{H}^{2}_{d+d^{\La}}(X),$$
and by Hodge theory one has that
$$\mathcal{H}^{2}_{d+d^{\La}}\subset\Om^{2}\cap\ker d=\mathcal{H}^{2}_{dR}\oplus d(\Om^{1}(X)).$$
Hence, one get
$$V=\mathcal{H}^{2}_{d+d^{\La}}(X)\cap d(\Om^{1}(X))$$
concluding the proof.
\end{proof}
\begin{remark}
There are some results on other degrees cases. In \cite[Theorem 4.3]{TT2}, it is proved that, if $(X^{2n},\w)$ is a $2n$-dimensional compact symplectic manifold (we don't need to suppose that $X^{2n}$ is almost K\"{a}herian), then $\De_{s}^{1}=0$.
\end{remark}
\section*{Acknowledgements}
We would like to thank the anonymous referee for careful reading of my manuscript and helpful comments. We would like to thank Prof. Tomassini for stimulating email discussions. This work was supported in part by NSF of China (11801539) and the Fundamental Research Funds of the Central Universities (WK3470000019), the USTC Research Funds of the Double First-Class Initiative (YD3470002002).

\end{document}